\documentclass{birkjour}

\usepackage{amsmath,amssymb,amsthm,color}
\usepackage[utf8]{inputenc}
\usepackage{amsopn}
\usepackage{latexsym}
\usepackage[all,cmtip]{xy}
\usepackage{tikz}
\usepackage{float}
\usepackage{url}

\theoremstyle{plain}
 \newtheorem{teo}{Theorem}[section]
  
 \newtheorem{pro}[teo]{Proposition}
 \newtheorem{cor}[teo]{Corollary}
 \newtheorem{lm}[teo]{Lemma}

 \newtheorem*{defi}{Definition} 
 
\theoremstyle{definition}
\newtheorem{ex}[teo]{Example} 
\newtheorem*{remark}{Remark}

\newcommand{\nc}{\newcommand}

 \nc{\ad}{\operatorname{ad}} 
 \nc{\Ad}{\operatorname{Ad}}
 \nc{\coad}{\operatorname{coad}} 
 \nc{\ct}{\operatorname{T}}
 \nc{\rank}{\operatorname{rank}} 
 \nc{\Irr}{\operatorname{Irr}}
 \nc{\End}{\operatorname{End}} 
 \nc{\Aut}{\operatorname{Aut}}
 \nc{\Inn}{\operatorname{Inn}} 
 \nc{\Der}{\operatorname{Der}}
 \nc{\Dera}{\operatorname{Dera}} 
 \nc{\Auto}{\operatorname{Auto}}
 \nc{\GL}{\operatorname{GL}}
 \nc{\SL}{\operatorname{SL}}
 \nc{\coord}{\operatorname{coord}}
 \renewcommand{\span}{\operatorname{span}}
 \nc{\codim}{\operatorname{codim}}
  \nc{\pr}{\operatorname{pr}}

 \newcommand{\R}{\mathbb R}
 \newcommand{\Q}{\mathbb Q}
\newcommand{\N}{\mathbb N}
\newcommand{\Z}{\mathbb Z}

\newcommand{\mg}{\mathfrak g }

\newcommand{\mn}{\mathfrak n }

\newcommand{\mk}{\mathfrak k }
\newcommand{\mc}{\mathfrak c }
\newcommand{\mv}{\mathfrak v }
\newcommand{\mh}{\mathfrak h }
\newcommand{\ma}{\mathfrak a }
\newcommand{\mgg}{\mathfrak g }

\renewcommand{\sl}{\mathfrak{sl} }

\newcommand{\tX}{ \tilde{X} }
\newcommand{\tY}{ \tilde{Y} }

\newcommand{\tomega}{ \tilde{\omega} }
\newcommand{\tOmega}{ \tilde{\Omega} }

\nc{\rmc}{\textrm{\rm C}}
\nc{\rmd}{\textrm{\rm D}}
\nc{\rad}{\textrm{\rm Rad}}

\newcommand{\lela}{\left \langle}
\newcommand{\rira}{\right \rangle}
\newcommand{\lra}{\longrightarrow}

\newcommand{\bs}{\backslash}

\newcommand{\X}{X+\xi}

\newcommand{\J}{\mathcal{J}}
\newcommand{\TM}{TE\oplus T^*E}

\newcommand{\Jj}{\mathcal{J}_{J}}
\begin{document}
\title[$T$-duality on nilmanifolds]{$T$-duality on nilmanifolds}

\author{Viviana del Barco}
\email{delbarc@ime.unicamp.br}

\address{UNR-CONICET (Argentina) and IMECC-UNICAMP (Brazil)}

\author{Lino Grama}
\email{linograma@gmail.com}
\address{IMECC-UNICAMP}

\author{Leonardo Soriani}
\email{leo.soriani@gmail.com}
\address{IMECC-UNICAMP}

 \date{\today}

\begin{abstract} 
We study generalized complex structures and $T$-duality (in the sense of Bouwknegt, Evslin, Hannabuss and Mathai) on Lie algebras and construct the corresponding Cavalcanti and Gualtieri map. Such a construction is called {\em Infinitesimal $T$-duality}. As an application we deal with the problem of finding symplectic structures on 2-step nilpotent Lie algebras. We also give a criteria for the integrability of the infinitesimal $T$-duality of Lie algebras to topological $T$-duality of the associated nilmanifolds.
\end{abstract}

\thanks{V. del Barco supported by FAPESP grant 2015/23896-5.\\
L. Grama supported by FAPESP grant 2016/22755-1 and 2012/18780-0.\\
L. Soriani  supported by FAPESP grant 2015/10937-5.}



\maketitle

\section{Introduction}

$\indent$ 

$T$-duality is an operation on $2$-dimensional conformal field theory sigma-models. It interchanges some geometric information in the target space and the resulting conformal field theory is equivalent. But, as done by Bouwknegt, Evslin, Hannabuss and Mathai in \cite{BEM} and \cite{BHM}, it is possible to study only the topological questions related to T-duality. This is our starting point.

The link between $T$-duality and generalized complex structures has been given by Cavalcanti and Gualtieri in \cite{CG}. They realized $T$-duality as an isomorphism between Courant algebroids, which depends of the choice of a closed $3$-form, of topologically distinct manifolds. Under certain conditions one can interchange ``complex'' and ``symplectic'' structures between $T$-dual manifolds; the integrability of these structures depends on the $3$-forms $H$, this is the reason of the quotation marks.

In this work we study $T$-duality in the context of nilmanifolds, with invariant $H$-flux. These are homogeneous compact manifolds associated to nilpotent Lie groups admitting lattices which carry a natural structure of torus bundles. Using algebraic constructions at Lie the algebra level, we are able to present, under certain conditions, an explicit construction of the $T$-dual of a nilmanifold. The $T$-dual is also a nilmanifold and the aforementioned restrictions are related to $H$ being an integral class. This invariant context permits us to obtain conclusive results and also to work explicitly with non-trivial 3-forms $H$. 

Our description of $T$-duality allows us to apply the results to the study of invariant symplectic structures on $2$-step nilpotent Lie groups following the spirit of the {\em mirror symmetry} program \cite{SYZ}: we understand the symplectic geometry of a manifold $E$ via the complex geometry of its mirror $E^\vee$, which, in our case is the $T$-dual manifold of $E$. The study of symplectic structures on nilpotent Lie groups (and in the corresponding nilmanifolds) is a very active topic in invariant geometry (see \cite{CdB,wang2013classification,RG} and references therein). An application in the context of generalized $G_2$-structures is presented in \cite{dBG}.

In this paper nilmanifolds are described as homogeneous spaces where $E=\Lambda\bs G$ and we consider invariant forms $H$ on them; here $G$ is a nilpotent Lie group and $\Lambda$ is a discrete cocompact subgroup. We address the questions of existence and constructions of $T$-duals, and also uniqueness of such. As usual, the invariant geometry of these homogeneous manifolds is evinced at the Lie algebra $\mgg$ of $G$. Because of their natural structure as torus bundles, nilmanifolds have already appeared as primary examples in the context of $T$-duality (see \cite{BHM06,CG,Mat} for instance). This work fully describes $T$-duality within this family, using their particular topology, algebraic and differential structure.

In the first part of the present paper, based on the works of Bouwknegt, Evslin, Hannabus and Mathai and of Cavalcanti and Gualtieri, we define a $T$-duality between Lie algebras which we call {\em Infinitesimal $T$-duality}. This is a general construction valid for any real Lie algebra with nontrivial center. We define the corresponding {\em Cavalcanti and Gualtieri map} establishing the one to one correspondence of generalized complex structures between dual Lie algebras. We notice that our methods differ from those in \cite{CLP} where an algebraic duality between Lie algebras is also considered; in fact our definition is independent of the existence of generalized complex structures on the Lie algebra.

In the second part of this work we deal with the question of the integrability of the {\em Infinitesimal $T$-duality}. This means, given two infinitesimal $T$-dual nilpotent Lie algebra $(\mathfrak{n}_1,H_1)$ and $(\mathfrak{n}_2,H_2)$, when is it possible to find lattices $\Lambda_1$ and $\Lambda_2$ in such way that the compact nilmanifolds $E_i:= \mathfrak{n}_i/\Lambda_i$, $i=1,2$ are topologically $T$-dual (in the sense of Bouwknegt, Evslin, Hannabuss and Mathai). We answer this positively: if $\mathfrak{n}_1$ admits a lattice and $H$ satisfies a rational condition, then $\mn_2$ also admits lattices and the $T$-duality of the nilmanifolds is guaranteed. We show examples where there is more than one lattice on $\mn_2$ giving the duality. In particular, we prove that for a given pair $(E,H)$, there is a family of non diffeomorphic manifolds $\{E_j\}_{j\in \N}$ which are $T$-dual to $(E,H)$, for some 3-forms.

\section{Preliminaries}

\subsection{Generalized complex structures}
Let us recall some standard facts about generalized complex geometry. For details see \cite{gualtieri}. Let $E$ be a smooth $n$-dimensional manifold and $H\in \varOmega^3(E)$  be a closed $3$-form.The sum of the tangent and cotangent bundle ${\TM}$ admits a natural symmetric bilinear form with signature $(n,n)$ defined by $$\langle X+\xi,Y+\eta \rangle=\frac{1}{2}(\eta(X)+\xi(Y)).$$ One can also define a bracket on the sections of ${\TM}$ (the Courant bracket) by $$[\X,Y+\eta]_H=[X,Y]+\mathcal{L}_X\eta-i_Yd\xi+i_Xi_YH.$$

\begin{defi}
A generalized complex structure on $(E,H)$ is an orthogonal automorphism $\J:\TM \to \TM$ such that  $\J^2=-1$ and its $i$-eigenbundle is involutive with respect to the Courant bracket.
\end{defi}

We remark that complex and symplectic structures are examples of ge\-neralized complex structures. If $J$ and $\omega$ are complex and symplectic structures respectively on $E$, then $$\J_J:=\left(\begin{array}{cc}
 -J & 0\\
  0 & J^* \end{array}\right)\ \mbox{and}\  \J_{\omega}:= \left(\begin{array}{cc}
0 & -\omega^{-1}\\
\omega & 0\end{array}\right)$$ are generalized complex structures on $E$ for $H=0$.
\medskip

When $E$ is a Lie group $G$ one can consider invariant generalized complex structures. In this case we assume $H$ to be a left invariant closed 3-form on $G$, which is identified with an alternating 3-form on the Lie algebra $\mg$ of $G$ and closed with respect to the Chevalley-Eilenberg differential. The Lie group also acts by left translations on $TG\oplus T^*G$
$$g\cdot (X+\xi)=(L_g)_*X+ (L_{g^{-1}})_*\xi.$$
A generalized complex structure on $(G,H)$ is said to be left invariant if it is equivariant with respect to this action. A left invariant generalized complex structure $\mathcal J$ is identified with the linear map it induces at the Lie algebra $J:\mgg\oplus\mgg^*\lra \mgg\oplus\mgg^*$ (see also \cite{ABDF}).  

Therefore, when restricting to the invariant context we endow $\mgg\oplus \mgg^*$ with the Courant bracket 
\begin{equation}\label{courantliealg}
[\X,Y+\eta]_H=[X,Y]+\iota_Xd\eta-i_Yd\xi+i_Xi_YH, \qquad X,Y\in \mgg,\; \xi,\eta\in \mgg^*.\end{equation}
We study linear maps $J:\mgg\oplus\mgg^*\lra \mgg\oplus\mgg^*$ satisfying $J^{2}=-1$, preserving the natural metric  
\begin{equation}\label{metriclie}
\lela \X,Y+\eta\rira=\frac12 (\eta(X)+\xi(Y)) \qquad X,Y\in \mgg,\; \xi,\eta\in \mgg^* \end{equation}
 and such that its $i$-eigenspace is involutive under the bracket \eqref{courantliealg}.
 
Notice that $\mgg\oplus \mgg^*$ with the bracket in \eqref{courantliealg} is a Lie algebra. When $H=0$ it is the semidirect product of $\mgg$ by $\mgg^*$ by the coadjoint action.
 
If $\Lambda$ is a discrete cocompact subgroup of $G$, any left invariant closed 3-form $H$ is induced to the quotient $\Lambda \bs G$. Moreover, any invariant generalized complex structure on $(G,H)$ descends to a generalized complex structure on $(\Lambda\bs G,H)$.

\subsection{Topological $T$-duality}\label{BHM}

Let us start with the definition of topological $T$-duality for torus bundle equipped with a closed $3$-form.
\begin{defi}
Let $E$ and $E^\vee$ be principal fiber bundles with  $k$-dimensional tori $T$ and $T^\vee$ as the fiber and over the same base $M$ and let $H \in \varOmega^3(E)$, $H^\vee \in \varOmega^3(E^\vee)$ be closed invariant $3$-forms. Let $E \times_M E^\vee$ be the product bundle and consider the diagram
$$\xymatrix{
    &(E \times_M E^\vee,p^*H-p^{\vee*}H^\vee) \ar[ld]^{p}  \ar[rd]_{p^{\vee}}& \\
(E,H)\ar[dr]&  & (E^{\vee},H^\vee).\ar[ld]\\
               & M  & }$$

We define $(E,H)$ and $ (E^{\vee},H^\vee)$ to be $T$-dual if $p^*H-p^{\vee*}H^\vee=dF$, where $F \in \varOmega^2(E \times_M E^\vee)$ is a $2$-form $T\times T^\vee$-invariant and non-degenerate on the fibers. The product $E \times_M E^\vee$ is called correspondence space.
\end{defi}

\begin{remark}
This is the definition of $T$-duality found in \cite{CG} and it is weaker than the one used in \cite{BEM,BHM,BRS,BS}. In the stronger definition the forms $H$, $H^\vee$ and $F$ represent integral cohomology classes and $F$ induces an isomorphism $H_1(T,\mathbb{Z})\to H^1(T^\vee,\mathbb{Z})$.
\end{remark}

Given a pair of $T$-dual torus bundles $(E,H)$ and $ (E^{\vee},H^\vee)$, Cavalcanti and Gualtieri in \cite{CG} define an isomorphism between the space of invariant sections
\begin{equation}
\varphi: (\TM)/T \to (TE^\vee\oplus T^*E^\vee)/T^\vee.
\end{equation}

Such isomorphism preserves the natural bilinear form and the Courant bracket twisted by the $3$-forms  $H$ and $\tilde{H}$. Using this isomorphism one can send $T$-invariant generalized complex structure from  $E$ to  $E^\vee$.

Using the 2-form $F$ on the correspondence space of $T$-dual pairs, the isomorphism $\varphi$ is given explicitly by 
\begin{equation}\label{fi}
\varphi(\X)= p^{\vee}_*(\hat{X})+p^*\xi-F(\hat{X}),
\end{equation}
where $\hat{X}$ is the unique lift of $X$ to $E \times_M E^\vee$ such that $p^*\xi-F(\hat{X})$ is a basic $1$-form. The existence and uniqueness is consequence of the non-degeneracy of $F$. 

\begin{teo}[Cavalcanti and Gualtieri, \cite{CG}]\label{fiso}
The map $\varphi$ is a isomorphism of Courant algebroids; that is, for all $u,v \in (\TM)/T$
$$\langle \varphi(u),\varphi(v) \rangle=\langle u,v \rangle\ \ \mbox{ and }\ \ [\varphi(u),\varphi(v)]_{H^\vee}= \varphi([u,v]_H).$$
\end{teo}

Using $\varphi$ one can transport invariant generalized complex structures between $T$-dual spaces:

\begin{cor}\label{coro}
Let $(E,H)$ and $(E^{\vee},H^\vee)$ be $T$-dual spaces. If $\J$ is an invariant generalized complex structure on $E$ then $$\tilde{\J}:=\varphi\circ \J \circ \varphi^{-1}$$ is an invariant generalized complex structure on $E^\vee$.
\end{cor}

\section{Infinitesimal $T$-duality}

\subsection{Basics on Lie algebras}

In this section we introduce basic notions on Lie algebras and the concept of central extension we shall use in the sequel of the paper. 

\begin{defi}
Let $\ma$ be an abelian Lie algebra and let $(\mn,[\,,\,])$ be a Lie algebra. Suppose a Lie algebra $(\mgg,[\,,\,]_\mgg)$ fits into an exact sequence of Lie algebra homomorphisms
$$0\lra \ma\stackrel{i}\lra\mgg\stackrel{q}\lra \mn\lra 0. $$ We say that $\mgg$ is a {\em central extension of $\mn$ by $\ma$} if $[i(\ma),\mgg]_\mgg=0$, i.e. $[i(x),y]_\mgg=0$ for all $x\in\ma$ and $y\in \mgg$. 
\end{defi}

If $\mgg$ is a central extension of $\mn$ by $\ma$, there exists a linear map $\beta:\mn\lra \mgg$ with $q\circ \beta=\operatorname{id}_\mn$. Moreover, the following is a well defined map $\Psi:\mn\times \mn\lra \ma$:
\begin{equation}\label{bracketcentralex}\Psi(x,y)=i^{-1}([\beta(x),\beta(y)]-\beta([x,y])),\quad x,y\in\mn. \end{equation}
This is a skew-symmetric bilinear form satisfying
\begin{equation}\label{eq.closed}\Psi([x,y],z)+\Psi([y,z],x)+\Psi([z,x],y)=0,\end{equation}
that is, $\Psi$ is a closed 2-form when considered the trivial representation \linebreak $\rho:\mn\lra \mathfrak{gl}(\ma)$. To make reference to this cocycle, $\mgg$ will be denoted by $\mn_\Psi$. 

From now on, we identify $\mn$ and $\ma$ with the vector subspaces $\beta(\mn)$ and $i(\ma)$, respectively. The homomorphism $q:\mgg\lra \mn$ also identifies $\mn$ with $\mgg/\ma$. In this context, $\mgg = \mn\oplus\ma$,  $\ma$ is a central ideal of $\mgg$ and the Lie bracket in $\mgg$ is related to that of $\mn$ by
\begin{equation}[x+z,x'+z']_\mgg =[x,x'] +\Psi(x,x'), \quad x,x'\in \mn, z,z'\in\ma.\label{eq.bracketextension}\end{equation}

Conversely, given $\ma$ and $\mn$ Lie algebras, $\ma$ abelian and a closed 2-form $\Psi$ on $\mn$ with values in $\ma$, the vector space $\mgg=\mn\oplus\ma$ with the Lie bracket in \eqref{eq.bracketextension} is a central extension of $\mn$ by $\ma$. It is well known that the central extensions of $\mn$, $\mn_\Psi$ and $\mn_{\Psi'}$ are isomorphic if and only if $\Psi-\Psi'=d\psi$ for some linear map $\psi:\mn\lra\ma$. We refer to \cite{SM,Sc} for basic facts about central extensions and Lie algebras in general.

Given a skew-symmetric bilinear form $\Psi:\mn\times\mn\lra \ma$ and a  basis $\{x_1,\ldots,x_m\}$  of $\ma$, there exist $f_1,\ldots,f_m\in \Lambda^2\mn^*$
such that $$\Psi(x,y)=\sum_{k=1}^m -f_k(x,y) x_k.$$ In this case we denote  $\Psi$ as $(f_1,\ldots,f_m)$ (we do not make reference to the basis unless needed).
It is clear that $\Psi$ is closed (see Eq. \eqref{eq.closed}) if and only if each $f_i$ is closed.

For a basis $\{x_1,\ldots,x_m\}$ of a Lie algebra, we denote with upper indices $\{x^1,\ldots,x^m\}$ the dual basis.
\begin{lm}\label{central}
Let $\mgg$ be a Lie algebra having a nontrivial central ideal $\ma$ and let $\{x_1,\ldots,x_m\}$ be a basis of $\ma$. Then $\mgg$ is isomorphic to the  central extension of $\mgg/\ma$ by the closed 2-form $(dx^1,\ldots,dx^m)$. 
\end{lm}

\begin{proof} Denote by $\mn$ the quotient Lie algebra $\mgg/\ma$. We identify $\Lambda^2\mn^*$ with the subspace $\{\alpha\in \Lambda^2\mgg^*:\iota_x\alpha=0\mbox{ for all }x\in\ma\}$, so that if $\alpha$ is inside that set, then it induces $\tilde\alpha$ where $\tilde\alpha(q(x),q(x))=\alpha(x,y)$ for any $x,y\in \mgg$. Analogous identification holds for $\Lambda^2\mn^*\otimes\ma$ as a subspace of $\Lambda^2\mgg^*\otimes\ma$.

Under this identification $d\eta\in\Lambda^2\mn^*$  for all $\eta\in \mgg^*$  since $\iota_xd\eta(y)=d\eta(x,y)=-\eta([x,y])=0$ for all $y\in \mgg$ and $x\in \ma$. In particular,  $dx^i\in\Lambda^2\mn^*$ for any basis $\{x_1,\ldots,x_m\}$ of $\ma$ and $(dx_1,\ldots,dx_m)$ is a closed 2-form in $\mn$ with values in $\ma$.

If $i:\ma\lra \mgg$ is the inclusion and $q:\mgg\lra\mn$ is the quotient map then the following is an exact sequence
$$0\lra \ma\stackrel{i}\lra\mgg\stackrel{q}\lra \mn\lra 0, $$ where $[i(\ma),\mgg]=0$, so $\mgg$ is a central extension of $\mn$ by $\ma$.
Fix a complement $\mv$ of $\ma$ in $\mgg$ so that $\mgg=\mv\oplus \ma$ and denote $\pr_\mv:\mgg\lra \mv$ and $\pr_\ma:\mgg\lra \ma$ the projections. Define $\beta:\mn\lra \mgg$ as $\beta(u)=x_u$ where $x_u\in \mv$ is the unique element such that $q(x_u)=u$, then $q\circ \beta=id_\mn$. The 2-form of this extension is $\Psi(u,v)=[\beta(u),\beta(v)]-\beta([u,v])$, $u,v\in \mn$. But $q([x_u,x_v])=[u,v]$, so
$$\Psi(u,v)=[x_u,x_v]-\pr_\mv[x_u,x_v]=\pr_\ma([x_u,x_v]).$$
We only have left to remark that $$(x,y)\in \mgg\times\mgg\mapsto\pr_\ma([x,y])=\sum_{k=1}^m -dx^k(x,y) x_k$$ is an element $\Lambda^2\mgg^*\otimes\ma$, lying inside  $\Lambda^2\mn^*\otimes\ma$.\end{proof}

\smallskip

The lower central series  $\{\rmc^j(\mgg)\}$ and the derived series  $\{\rmd^j(\mgg)\}$ of a Lie algebra $\mgg$ are defined for all $j\geq 0$ by
$$\rmc^0(\mgg)=\mg,\quad  \rmc^j(\mgg)=[\mg,\rmc^{j-1}(\mgg)],\;\; j\geq 1.
$$
$$\rmd^0(\mgg)=\mg,\quad  \rmd^j(\mgg)=[\rmd^{j-1}(\mgg),\rmd^{j-1}(\mgg)],\;\; j\geq 1,
$$
We notice that $\rmc^1(\mgg)=[\mg,\mg]=\rmd^1(\mgg)$ is the commutator of $\mg$ and $\rmd^j(\mgg)\subseteq \rmc^j(\mgg)$ for all $j\geq 0$. A Lie algebra $\mgg$ is $j$-step solvable if $\rmd^{j}(\mg)=0$ and $\rmd^{j-1}(\mg)\neq 0$ for some $j\geq 0$. A solvable Lie algebra is said to be $k$-step nilpotent if $\rmc^k(\mgg)=0$ while $\rmc^{k-1}(\mgg)\neq 0$. 

\begin{pro}\label{pro.solucentext}
Let $\mgg$ be the central extension of a Lie algebra $\mn$. Then $\mgg$ is solvable (resp. nilpotent) if and only if $\mn$ is solvable (resp. nilpotent).
\end{pro}
\begin{proof}
Since $\mn$ is a quotient of $\mgg$ by an ideal, it is clear that $\mn$ is solvable or nilpotent if $\mgg$ is so.
For the converse, use the following inclusions for $k\geq 1$
$$\rmc^k(\mgg)\subseteq \rmc^k(\mn)+\Psi(\mn,\rmc^{k-1}(\mn)),\quad  \rmd^k(\mgg)\subseteq \rmd^k(\mn)+\Psi(\rmd^{k-1}(\mn),\rmd^{k-1}(\mn)).$$
These can be proved by a standard induction procedure.
\end{proof}
Notice that the steps of nilpotency or solvability of $\mgg$ are at most one more than that of $\mn$.

A Lie algebra is semisimple if its Killing form is nondegenerate. In particular, it coincides with its commutator and has no nontrivial abelian ideals. From these facts, it is clear that the central extension of a semisimple Lie algebra is never semisimple.

\subsection{Dual Lie algebras}
In this section we work with pairs of Lie algebras $\mgg$ and $\mgg^\vee$ that are isomorphic, up to a quotient by abelian ideals. That is, there exist abelian ideals  $\ma$ and $\ma^\vee$ in $\mgg$ and $\mgg^\vee$ such that $\mgg/ \ma \simeq\mgg^\vee/\ma^\vee$. In this case we denote $\mn$ the quotient Lie algebra and $q:\mgg\lra \mn$  and $q^\vee:\mgg^\vee\lra \mn$ the quotient maps.

The subspace $\mc$ of $\mgg\oplus\mgg^\vee$
$$\mc=\{(x,y)\in \mgg\oplus\mgg^\vee: q(x)=q^\vee(y)\} $$ is a Lie subalgebra and the following diagram is commutative
$$
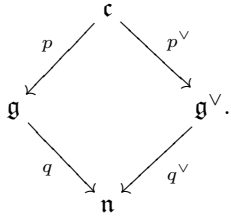
\xymatrix{
    &\mc \ar[ld]_{p}  \ar[rd]^{p^\vee}& \\
\mgg\ar[dr]_{q}&  & \mgg^\vee\ar[ld]^{q^\vee}.\\
               & \mn  & }
               $$
Here $p$ and $p^\vee$ are the projections over the first and second component, respectively. The Lie subalgebras $\mk=\{(x,0)\in\mc:x\in \ma\}$ and $\mk^\vee=\{(0,y)\in\mc:y\in \ma^\vee\}$ are also abelian ideals of $\mc$. In particular $\mc/\mk^\vee\simeq \mgg$ and $\mc/\mk\simeq \mgg^\vee$. As a vector space, $\mc$ is isomorphic to $\mn\oplus\ma \oplus \ma^\vee$.

A 2-form $F\in \Lambda^2\mc^*$ is said to be {\em non-degenerate in the fibers} if for all $x\in \mk$, there exists some $y\in \mk^\vee$ such that $F(x,y)\neq 0$. Such an $F$ exists if and only if $\dim \ma=\dim\ma^\vee$.

Assume $F$ is a non-degenerate 2-form in $\mc$ and let $x\in \mgg$ and $\xi \in \mgg^*$. Choose $y_0\in \mgg^\vee$ such that $q(x)=q^\vee(y_0)$, then ${p^{-1}(x)=\{(x,y_0+z): z\in \ma^\vee\}}$. There exists a unique $z_0\in \ma^\vee$ such that $\left(p^*\xi-F((x,y_0),\cdot) \right)|_{\mk}=F((0,z_0),\cdot)|_{\mk}$. Denote $u_x=(x,y_0+z_0)\in \mc$, then we have that $p^*\xi-F(u_x,\cdot) $ annihilates on $\mk$ so it is the pullback of a 1-form in $\mgg^\vee$.  Notice that $u_x$ does not depend on the choice of $y_0$.
 We shall define $\sigma_{\xi}\in\mgg^{\vee*}$ such that 
\begin{equation}\label{sigmadef}
p^*\xi-F(u_x,\cdot) =p^{\vee*}\sigma_\xi.
\end{equation}

\smallskip

The duality of Lie algebras we introduce below, corresponds to an infinitesimal version of the $T$-duality of principal torus bundles introduced in the previous section.

Let $\mgg$ be a Lie algebra together with a closed 3-form $H$. Let $\ma$ be an abelian ideal of  $\mgg$, we say that the triple $(\mgg,\ma,H)$ is admissible if $H(x,y,\cdot)=0$ for all $x,y\in\ma$. Notice that when $\dim \ma=1$ then any closed 3-form gives an admissible triple.  

\begin{defi} Two  admissible triples $(\mgg,\ma,H)$ and $(\mgg^\vee,\ma^\vee,H^\vee)$ are said to be {\em dual} if $\mgg/ \ma \simeq\mn\simeq \mgg^\vee/\ma^\vee$ and there exist a 2-form $F$ in $\mc$ which is non-degenerate in the fibers such that $p^*H-p^{\vee*}H^\vee=dF$.
\end{defi}

\begin{ex}
Let $\mg$ be the $n$ dimensional abelian Lie algebra and $\ma$ any $m$ dimensional proper subspace.  Therefore $(\mg,\ma,H)$ is dual to itself if and only if $H$ is a basic form, that is, it is a pullback from a form on $\mg/\ma$. Notice that for any $F \in \Lambda^2 \mc^*$ we have $dF=0$.
\end{ex}

In some cases we will say that $\mgg$ and $\mgg^\vee$ are dual meaning that there exist $H,H^\vee, \ma,\ma^\vee$ such that $(\mgg,\ma,H)$ and $(\mgg^\vee,\ma^\vee,H^\vee)$ are dual admissible pairs. We prove existence and uniqueness of dual triples.

\begin{teo}\label{teo.infdual} Let $(\mgg,\ma,H)$ be an admissible triple with $\ma$ a central ideal and let $\{x_1,\ldots,x_m\}$ be a basis  of  $\ma$. Let $\ma^\vee=\R^m$ and define  
\begin{itemize}
\item $\Psi^\vee:\mn\times\mn\lra \ma^\vee$ given by $\Psi^\vee=(\iota_{x_1}H,\ldots,\iota_{x_m}H)$, 
\item $\mgg^\vee=(\mgg/\ma)_{\Psi^\vee}$ and
\item $H^\vee=\sum_{k=1}^m z^k\wedge dx^k+\delta$ where $\{z_1,\ldots,z_m\}$ is a basis of $\ma^\vee$ and $\delta$ is the basic component of $H$.
\end{itemize}
Then $(\mgg^\vee,\ma^\vee,H^\vee)$ is an admissible triple and is dual to $(\mgg,\ma,H)$. 

Conversely, if $(\mgg^\vee,\ma^\vee,H^\vee)$ is dual to $(\mgg,\ma,H)$, then there exist a basis $\{x_1,\ldots,x_m\}$ of  $\ma$ and a basis $\{z_1,\ldots,z_m\}$ of $\ma^\vee$ such that the formulas above hold.
\end{teo}

\begin{proof}

For each $k=1,\ldots,m$ define $\Psi_k^\vee=\iota_{x_k}H$. The condition $H(x,x',\cdot)=0$ for all $x,x'\in \ma$  implies $\iota_{x}\Psi_k^\vee=0$ for all $x\in \ma$ and $k=1, \ldots, m$. Define $\delta=H-\sum_{k=1}^m\Psi_k^\vee\wedge x^k$, then $\iota_{x}\delta=0$ for all $x\in \ma$.

Let $G$ be a connected Lie group with Lie algebra $\mg$. Then $x_k$ (resp. $H$) defines a left-invariant vector field (resp. 3-form) on $G$. Since $x_k$ is in the center of $\mg$, left and right translations by $\exp t X$ coincide so the Lie derivative of $H$ along $x_k$ is zero, that is, $\mathcal L_{x_k}H=0$. Cartan's magic formula and $dH=0$ lead us to $d\Psi_k^\vee=d\iota_{x_k}H=0$. 

Because of the definition of $\delta$, $d\delta=dH-\sum_{k=1}^m d\Psi_k^\vee\wedge x^k -\sum_{k=1}^m \Psi_k^\vee\wedge dx^k$. But $dH=0=d\Psi_k ^\vee$ so we conclude
\begin{eqnarray}
0&=&d\Psi_k^\vee,\quad k=1,\ldots,m\label{eq.psik}\\
d\delta&=&-\sum_{k=1}^m \Psi_k^\vee\wedge dx^k.\label{eq.ddelta}
\end{eqnarray}

The fact that $\iota_{x}\Psi_k^\vee=\iota_{x}\delta=0$ for any  $x\in\ma$ implies that $\Psi_k^\vee$ and $\delta$ can be defined as forms in $\mn$. That is, the values of these 2- and 3- forms is constant along the equivalence classes in $\mg/\ma$. 

We mantain the same notation for these forms induced in $\mn$. Equation \eqref{eq.psik} implies that $\Psi^\vee=(\iota_{x_1}H,\ldots,\iota_{x_m}H)$ is closed in $\mn$ so one can consider the central extension of $\mn$  by $\Psi^\vee$, which we shall denote $\mgg^\vee$.

The first equation implies that $\Psi^\vee:\mn\times\mn\lra \R^m$ defined by components as $\Psi^\vee=(\iota_{x_1}H,\ldots,\iota_{x_m}H)$ is closed. So one can consider the central extension of $\mn$  by $\Psi^\vee$, which we shall denote $\mgg^\vee$. The central ideal appearing in this central extension is $\ma^\vee=\R^m$; denote $\{z_1,\cdots,z_m\}$ the basis of $\ma^\vee$ such that $dz^k=\Psi_k^\vee$.

Set $H^\vee=\sum_{k=1}^m z^k\wedge  dx^k+\delta$. Then
$dH^\vee=\sum_{k=1}^m dz^k\wedge dx^k+d\delta=\sum_{k=1}^m \Psi_k^\vee\wedge dx^k+d\delta=0$ in virtue of \eqref{eq.ddelta}, so $H^\vee$ is closed. By definition $H^\vee(z_i,z_j,\cdot)=0$ so the triple $(\mgg^\vee,\ma^\vee,H^\vee)$ is an admissible triple.

Notice  that  $dx^k$, $dz^k$ and $\delta$ are 2-forms in $\mn$ so their pullbacks by $p$ and $p^\vee$ coincide. In the correspondence space $\mc$ consider the 2-form $F=\sum_{k=1}^m p^{\vee*}z^k\wedge p^*x^k $, which is non-degenerate in the fibers and it satisfies
\begin{eqnarray*}
p^*H-p^{\vee*}H^\vee&=&p^*\left( \sum_{k=1}^m dz^k\wedge x^k +\delta\right)- p^{\vee*}\left(\sum_{k=1}^m z^k\wedge  dx^k+\delta\right)\\
&=& \sum_{k=1}^m p^*dz^k\wedge p^*x^k - \sum_{k=1}^m p^{\vee*}z^k\wedge  p^{\vee*}dx^k\\
&=& \sum_{k=1}^m p^{\vee*}dz^k\wedge p^*x^k - \sum_{k=1}^m p^{\vee*}z^k\wedge  p^{*}dx^k\\
&=& d\left(\sum_{k=1}^m p^{\vee*}z^k\wedge p^*x^k \right)=dF.
\end{eqnarray*}
Therefore the triples are indeed dual triples.
\smallskip

Now we prove the converse. Assume $(\mgg^\vee,\ma^\vee,H^\vee)$ is dual of $(\mgg,\ma,H)$, then $\mgg^\vee$ has a central ideal $\ma^\vee$ such that $\mgg^\vee/\ma^\vee\simeq \mn\simeq \mgg/\ma$ and $\mgg^\vee$ is the central extension of $ \mn$ by a closed 2-form $\Psi^\vee$.

The 2-form $F\in \Lambda^2\mc^*$ given by the duality restricts to a non-degenerate form $F:\mk\times \mk^\vee\lra \R$. Let $\{z_1,\ldots,z_m\}$ a basis of $\ma^\vee$ and $\Psi^\vee=(dz^1,\ldots,dz^m)$, then $\mgg^\vee=\mn_{\Psi^\vee}$ by Lemma \ref{central}. 

For each $k=1,\ldots,m$ denote $\tilde z_k=(0,z_k)\in\mk^\vee$ and notice that $d\tilde z^k=p^{\vee*}dz^k$, then there exists $\tilde x_k=(x_k,0)\in\mk$ such that $F(\cdot,\tilde x_k)=\tilde z^k$. Clearly, $\{x_1,\ldots,x_m\}$ is a basis of $\ma$. Moreover
$\iota_{\tilde x_k}dF=\iota_{\tilde x_k}p^*H=p^*\iota_{x_k}H$ but at the same time $\iota_{\tilde x_k}dF=-d\iota_{\tilde x_k}F$ because $\tilde x_k$ is central, therefore $p^{\vee*}dz^k=p^{*}\iota_{x_k} H$. The 3-form $H$ being admissible for $\ma$ implies that $\iota_{x_k}H$ is basic and so is $dz^k$, so the previous equality implies $\iota_{x_k}H=dz^k$ and hence $\Psi^\vee=(\iota_{x_1}H,\ldots,\iota_{x_m}H)$. Following similar steps as in the first part of the proof we obtain $\Psi^\vee=(\iota_{x_1}H,\ldots,\iota_{x_m}H)$.\end{proof}

Duality is closed in the family of solvable and nilpotent Lie algebras.
\begin{cor} If $\mgg$ and $\mgg^\vee$ are dual then $\mgg$ is solvable if and only if $\mgg^\vee$ is solvable. Moreover, $\mgg$ is nilpotent if and only if $\mgg^\vee$ is so.
\end{cor}
\begin{proof}
The Lie algebras $\mgg$ and $\mgg^\vee$ are central extensions of the same Lie algebra $\mn$, so the result follows from Proposition \ref{pro.solucentext}.
\end{proof}

\begin{cor}\label{toro2step} For a Lie algebra $\mgg$ and a central ideal $\ma$, the triple  $(\mgg,\ma, H=0)$ is admissible and the dual $\mgg^\vee$ satisfies $\mgg^\vee\simeq \mn\oplus \R^m$ as a Lie algebra. In particular if $\mgg$ is 2-step nilpotent and $\ma$ contains the commutator of $\mgg$ then $\mgg^\vee$ is an abelian Lie algebra and $H^\vee\neq 0$ 
\end{cor}

\begin{ex}
One can specify a Lie algebra $\mg$ by listing the derivatives of a basis $\{e^1,\dots,e^n\}$ of $\mg^*$ as an $n$-uple of $2$-forms $(de^k=\sum c_{ij}^ke^i\wedge e^j)_{k=1}^n$. To simplify the notation we write $e^{ij}$ for the $2$-form $e^i\wedge e^j$. This is the Malcev's notation for nilpotent Lie algebras. For example, the $6$-uple $(0,0,0,e^{12},e^{13},e^{14})$ is the Lie algebra with dual generated by $e^1,\dots,e^6$ such that $de^1=de^2=de^3=0$, $de^4=e^1 \wedge e^2$, $de^5=e^1 \wedge e^3$ and $de^6=e^1 \wedge e^4$. This notation is very useful to explicit the dual of a given admissible triple.

Let $\mg=(0,0,0,e^{12},e^{13},e^{14})$, $\ma=\langle e_5,e_6\rangle$ and $H=e^{123}+e^{135}+e^{246} \in \Lambda^3\mg^*$. The triple $(\mg,\ma,H)$ is admissible. In this case $\Psi^\vee=(e^{13},e^{24})$,  ${\mg/\ma = (0,0,0,e^{12})}$ and $\delta= e^{123} $. Then $\mg^\vee=(0,0,0,e^{12},e^{13},e^{24})$ and $H^\vee=e^{123}+e^{135}+e^{146}$.
\begin{center}
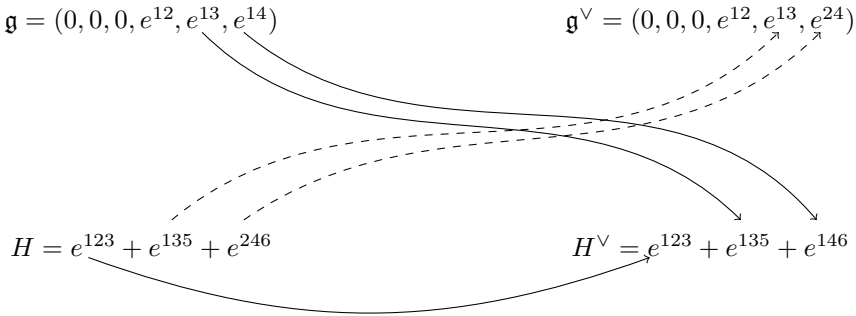
\begin{figure}[H]
\begin{tikzpicture}
\node at (1,0) {$H=e^{123}+e^{135}+e^{246}$};
\node at (8.5,0) {$H^\vee=e^{123}+e^{135}+e^{146}$};
\node at (1,3) {$\mg=(0,0,0,e^{12},e^{13},e^{14})$};
\node at (8.5,3) {$\mg^\vee=(0,0,0,e^{12},e^{13},e^{24})$};
\draw [->] (1.8,2.85) to [out=-45, in=135](8.9,0.35);
%\draw [->] (2.35,2.85) to [out=-90,in=180] (6,1.4) to [out=0,in=90] (11.4,0.35);
\draw [->] (2.35,2.85) to [out=-40,in=130] (9.9,0.35);
\draw [dashed, ->] (1.4,0.35) to [out=40,in=-135](9.4,2.85);
\draw [dashed, ->] (2.35,0.35) to [out=35,in=-135](9.95,2.85);
\draw [->] (0.3,-0.15) to [out=-20,in=-160](7.7,-0.15);

\end{tikzpicture}
\caption{The diagram shows how to construct the dual  of a given admissible triple.}
\end{figure}
\end{center}
\end{ex}

\begin{teo} Let $\mgg$, $\mgg^\vee$ be Lie algebras and let $H,H^\vee$ be closed 3-forms in  $\mgg$, $\mgg^\vee$, respectively. If there exist abelian ideals $\ma$, $\ma^\vee$ such that $(\mgg,\ma,H)$ and $(\mgg^\vee,\ma^\vee,H^\vee)$ are admissible dual triples, then there exists an isomorphism $\varphi:\mgg\oplus\mgg^*\lra \mgg^\vee\oplus\mgg^{\vee*}$ preserving the Courant bracket \eqref{courantliealg} and the canonical bilinear form \eqref{metriclie}.
\end{teo}

\begin{proof} Let $x+\xi\in \mgg\oplus\mgg^*$. As discussed before (see Eq.\eqref{sigmadef}) nondegeneracy of $F$ implies that there exist  unique $u_x\in \mc$  and $\sigma_\xi\in \mgg^{\vee*}$ such that
\begin{equation}\label{eq.sigmaxi}
p\, u_x=x\;\mbox{ and }\; p^*\xi-F(u_x,\cdot) =p^*\sigma_\xi.
\end{equation}
Thus we can define $\varphi:\mgg\oplus\mgg^*\lra \mgg^\vee\oplus\mgg^{\vee*}$ as
$x+\xi\mapsto \varphi(x+\xi)= p^\vee u_x+\sigma_\xi$.  It is easy to check that $\varphi$ is a linear isomorphism, moreover for any $x,y\in \mgg$ and $\xi,\eta\in \mgg^*$ we have
\begin{eqnarray*}
\lela \varphi(x+\xi),\varphi(y+\eta)\rira &=&
\lela p^\vee u_x+\sigma_\xi,p^\vee u_y+\sigma_\eta\rira\\
&=& \frac{1}{2} (\sigma_\xi(p^\vee u_y)+\sigma_\eta(p^\vee u_x))\\
&=& \frac{1}{2} (p^*\xi(u_y)-F(u_x,u_y)+p^*\eta(u_x)-F(u_y,u_x))  \\
&=& \frac{1}{2} (\xi(y)+\eta(x)).
\end{eqnarray*}

In order to show that $\varphi$ behaves well with the Courant bracket, that is, for $x,y\in \mgg$ $\xi,\eta\in \mgg^*$
\begin{equation}\label{eq.phibracket}
\varphi([x+\xi,y+\eta]_H)=[\varphi(x+\xi),\varphi(y+\eta)]_{H^\vee},
\end{equation}
we analyze separately the vector and 1-form parts. From the definitions of the Courant bracket and $\varphi$, Eq. \eqref{eq.phibracket} holds if and only if 
\begin{eqnarray}
p^\vee u_{[x,y]}&=&[p^\vee u_x,p^\vee u_y] \mbox{ and }\label{eq.part1}\\
\alpha &=& i_{p^\vee u_x}d\sigma_\eta-i_{p^\vee u_y}d\sigma_\xi+i_{p^\vee u_x}i_{p^\vee u_y}H^\vee \label{eq.part2}
\end{eqnarray}
where $p^*(i_xd\eta-i_yd\xi+i_xi_yH)-F(u_{[x,y]},\cdot)=p^{\vee*}\alpha$.

It is clear that $p([u_x,u_y])=[x,y]$, moreover
for $(z,0)\in \mk$ we have
\begin{eqnarray*}
&& p^*(i_xd\eta-i_yd\xi+i_xi_yH)(z,0)-F([u_x,u_y],\cdot)(z,0)\\
&&\hspace{1.5cm}=-\eta([x,z])+\xi([y,z])+H(x,y,z)-F([u_x,u_y],(z,0))\\
&&\hspace{1.5cm}=p^*H(u_x,u_y,(z,0))-dF(u_x,u_y,(z,0))\\
&&\hspace{1.5cm}=H^\vee(p^\vee u_x,p^\vee u_y,p^\vee(z,0))\\
&&\hspace{1.5cm}= 0.
 \end{eqnarray*} Here we have used the fact that $\mk$ is in the center of $\mc$ and $dF=p^*H-p^{\vee *} H^\vee$. We conclude that $u_{[x,y]}=[u_x,u_y]$ and thus Eq. \eqref{eq.part1} holds.

We shall prove that $p^{\vee*}\alpha = p^{\vee*}\left(i_{p^\vee u_x}d\sigma_\eta-i_{p^\vee u_y}d\sigma_\xi+i_{p^\vee u_x}i_{p^\vee u_y}H^\vee\right)$. Since $p^{\vee}$ is surjective, Eq. \eqref{eq.part2} will hold.

Notice that $p^{\vee*}(i_{p^\vee u_x}d\sigma_\eta)=\iota_{u_x}(p^{\vee*}d\sigma_\eta)$ and analogous equality holds for the other forms involved, thus
\begin{eqnarray*}
&&p^{\vee*}\left(i_{p^\vee u_x}d\sigma_\eta-i_{p^\vee u_y}d\sigma_\xi+i_{p^\vee u_x}i_{p^\vee u_y}H^\vee\right)\\ 
&&\hspace{1.5cm}=i_{ u_x}dp^{\vee*}\sigma_\eta-i_{u_y}dp^{\vee*}\sigma_\xi+i_{u_x}i_{ u_y}p^{\vee*} H^\vee\\
&&\hspace{1.5cm}=i_{ u_x}d(p^*\eta-\iota_{u_y}F)-i_{u_y}d(p^*\xi-\iota_{u_x}F)+i_{u_x}i_{ u_y}p^{\vee*} H^\vee\\
&&\hspace{1.5cm}=p^*i_{x}d\eta-p^*i_{y}d\xi-
\iota_{u_x}d\iota_{u_y}F -\iota_{u_y}d\iota_{u_x}F
+i_{u_x}i_{ u_y}(p^*H-dF)\\
&&\hspace{1.5cm}=p^*i_{x}d\eta-p^*i_{y}d\xi+i_{u_x}i_{ u_y}p^*H-\iota_{[u_x,u_y]}F.\end{eqnarray*}\end{proof}

As in the case of global $T$-duality (see Theorem \ref{fiso}), we conclude that the map $\varphi$ is an isomorphism of the Courant algebroids structures on the Lie algebras, so we have a bijection between generalized complex structures on dual Lie algebras.

\begin{cor}\label{transfer}
Let $(\mgg, \ma,H)$ and $(\mgg^\vee,\ma^\vee,H^\vee)$ be dual triples. If $J$ is a generalized complex structure on $(\mgg,H)$ then $$\tilde{J}:=\varphi\circ J \circ \varphi^{-1}$$ is an invariant generalized complex structure on $(\mgg^\vee,H^\vee)$.
\end{cor}

\begin{remark} Given two infinitesimal $T$-dual Lie algebras, consider the basis $\{y_1,\ldots,y_t,x_1,\ldots,x_m\}$ and $\{y_1,\ldots,y_t,z_1,\ldots,z_m\}$ as in Theorem \ref{teo.infdual}. The matrix of the isomorphism $\varphi$ on these basis and the corresponding duals has the form $$\varphi=\left(\begin{array}{cccc}
1_{t\times t} & 0 & 0 & 0\\
0 & 0 & 0 & -1_{m\times m}\\
0 & 0 & 1_{t\times t} & 0\\
0 & -1_{m\times m} & 0 & 0
\end{array}\right). $$
\end{remark}

\begin{ex}\label{KTT}
Let $\mh_3$ be the Heisenberg Lie algebra: it has a basis $\{e_1,e_2,e_3\}$ such that the only non-zero bracket is $[e_1,e_2]=e_3$. Consider $\mg=\mh_3 \oplus \mathbb{R} $ and ${\ma =\langle e_3 \rangle \oplus \mathbb{R} \subset \mg}$. As observed in Corollary \ref{toro2step}, $(\mg,\ma, 0)$ is dual to  $(\ma_4,\ma_2,H^\vee)$, where $\ma_i$ is the abelian algebra of dimension $i$ and $H^\vee \neq 0$. Explicitly, using the Malcev notation: $\mg=(0,0,e^{12},0)$, $\ma_4=(0,0,0,0)$ and $H^\vee=e^{123}$.

The symplectic structures on $\mg$ are, according to \cite{Ovando}, of the form 
$$\omega=a_{12}e^{12}+a_{13}e^{13} + a_{14} e^{14}+a_{23}e^{23} + a_{24} e^{24},$$ $$\mbox{with}\ a_{14}a_{23}-a_{13}a_{24} \neq 0. $$ 

We regard them as generalized complex structures and use Corollary \ref{transfer} to transport them to generalized complex structures on $\mathfrak{a}_4$. If $a_{12}=0$ the resulting structure is:
$$J=\left(\begin{array}{cc}
0 & Y^{-1} \\ 
-Y & 0
\end{array}\right)\ \mbox{with} \ Y=\left(\begin{array}{cc}
a_{13} & a_{23} \\ 
a_{14} & a_{24}
\end{array}\right).$$ This endomorphism satisfies $J^2=-1$ and the generalized complex structure $\mathcal J_J$ it induces is integrable with respect to $H^\vee$ and is of type $2$ (see \cite{CG} for notion of type), that is, it is generalized complex of complex type. Precisely, $J$ is $H^\vee$-integrable in accordance to the definition we give in Subsection \ref{app}.

In the next table we write explicitly the correspondence between symplectic and {$H^\vee$-integrable} complex structures.

\begin{center}
\begin{table}[ht]
\begin{tabular}{|c|c|}
\hline
Symplectic structure on $\mg$ & $H^\vee$-integrable complex structure on $\ma_4$\\ 
\hline 
$a_{13}e^{13} + a_{14} e^{14}+a_{23}e^{23} + a_{24} e^{24}$ &  $J=\left(\begin{array}{cc}
0 & Y^{-1} \\ 
-Y & 0
\end{array}\right)$, $Y=\left(\begin{array}{cc}
a_{13} & a_{23} \\ 
a_{14} & a_{24}
\end{array}\right)$\\ 
\hline
\end{tabular} 
\end{table}
\end{center}
\end{ex}

\subsection{Applications: symplectic structures on $2$-step nilpotent Lie algebras}\label{app}

As seen in Corollary \ref{toro2step}, if $\mg$ is $2$-step nilpotent and $[\mg,\mg]\subset \ma$, the dual of $(\mg,\ma,0)$ is $(\mg^\vee,\ma^\vee,H^\vee)$, where $\mg$ is the abelian algebra and $H^\vee\neq 0$. If, additionally, we have that $\mg$ is $2n$-dimensional and $\ma$ is $n$-dimensional, the symplectic structures of $\mg$ such that $\ma$ is Lagrangian are transported (via $\varphi$) to  complex structures in $\mg^\vee$ such that $\ma^\vee$ is real (this was already observed in \cite{CG}). This is exactly the situation of the Example \ref{KTT} above.

Using this idea, to look for symplectic structures of this kind on $2$-step nilpotent algebra is the same thing than to look for complex structures on abelian algebras. But since $H^\vee\neq 0$, these complex structures are not the usual ones, they need to be compatible with $H^\vee$ in some sense. In the following we explain this compatibility.

Let $J$ be an almost complex structure on a Lie algebra $\mg$, $H\in \Lambda^3\mg^*$ and consider $\Jj: \mg \oplus \mg^* \to \mg \oplus \mg^*$ $$\Jj =\left(\begin{array}{cc}
-J & 0\\
0 & J^*\end{array}\right).$$ $\Jj$ is orthogonal and satisfies $\Jj^2=-1$. Suppose its $i$-eigenspace is involutive with respect to the Courant bracket twisted by $H$. This is equivalent to the annihilation of the ``Nijenhuis tensor'' defined using the Courant bracket:
$$0 = \Jj[\cdot,\cdot]_H - [\Jj(\cdot),\cdot]_H - [\cdot,\Jj(\cdot)]_H - \Jj[\Jj(\cdot),\Jj(\cdot)]_H$$

Plugging in vectors $x,y \in \mg$ and separating vectors and $1$-forms we get
$$\left\{\begin{array}{ll}
-J[x,y]+[Jx,y]+[x,Jy]+J[Jx,Jy]=0\\
J^*(i_xi_yH)+i_{Jx}i_yH+i_xi_{Jy}H-J^*(i_{Jx}i_{Jy}H)=0.
\end{array}\right.$$
The first equation is the usual integrability condition of complex structures. The second one, when we plug in a third vector $z \in \mg$, is this rather nice and symmetrical condition:
\begin{equation}\label{HJ}
H(Jx,y,z)+H(x,Jy,z)+ H(x,y,Jz)=H(Jx,Jy,Jz)\ \ \ \forall x,y,z \in \mg.
\end{equation}

One can show that this necessary condition for the involutivity  of the $i$-eigenspace of $\Jj$ is also sufficient.

\begin{defi}
An almost complex structure is called $H$-integrable if it is integrable and satisfies (\ref{HJ}).
\end{defi}

In order to produce a class of symplectic $2$-step nilpotent Lie algebras we will fix a complex structure $J$ on the abelian algebra and check for which $H$ this $J$ is $H$-integrable. For each of these $H$ we can build the dual $2$-step nilpotent algebra, which has an invariant symplectic structure: $\varphi\circ \mathcal J_J \circ \varphi^{-1}$ (see Corollary \ref{transfer}).  

Let $\mg$ be the $2n$-dimensional abelian Lie algebra with basis $\{e_1,\dots, e_{2n}\}$ and $\ma=\langle e_{n+1},\dots,e_{2n} \rangle$. Let $J:\mg\to \mg$ be the complex structure given by $Je_i=e_{n+i}$ for $i=1,\dots,n$.

Let's check for which $H\in \Lambda^3 \mg^*$ $J$ is $H$-integrable. For $i\leq n$ the equation (\ref{HJ}) with $x=e_i, y=e_{n+i}, z=e_k$ boils down to $$-H(e_{n+1},e_i,Je_k)=H(e_i,e_{n+1},Je_k), $$ which is not an extra condition on $H$, since it is already skew-symmetric. This kind of redundancy happens every time we pick triples $e_i,e_j,e_k$ such that two of them are related by $J$.

Now take $i,j,k\in \{1 \dots,n\}$ all distinct (since equation (\ref{HJ}) holds trivially if, for instance, $x=y$). For $x=e_i, y=e_j, z=e_k$ we get
\begin{equation}\label{HJ2}
H(e_{n+i},e_{n+j},e_{n+k})=H(e_{n+i},e_j,e_k)+H(e_i,e_{n+j},e_k)+ H(e_i,e_j,e_{n+k})
\end{equation}
For $x=Je_i=e_{n+i}, y=e_j, z=e_k$ we get
\begin{equation}\label{HJ3}
H(e_i,e_j,e_k)=H(e_i,e_{n+j},e_{n+k})+H(e_{n+i},e_j,e_{n+k})+ H(e_{n+i},e_{n+j},e_k)
\end{equation}
All the other triples give restrictions equivalent to one of these. 

For us to be able to build the dual algebra, $(\mg,\ma,H)$ must be an admissible triple. This implies that everything vanishes on equation (\ref{HJ3}) and equation (\ref{HJ2}) becomes
\begin{equation}\label{HJ4}
H(e_{n+i},e_j,e_k)+H(e_i,e_{n+j},e_k)+ H(e_i,e_j,e_{n+k})=0.
\end{equation}
We have $\binom{n}{3}$ equations like this one. We summarize the discussion above in the next proposition.

\begin{pro}
The complex structure $J$ is $H$-integrable if and only if it satisfies equation (\ref{HJ4}) for all $i,j,k$.
\end{pro}

\begin{remark}
One can start with a different complex structure and do the same calculations above to get a condition similar to (\ref{HJ4}).
\end{remark}

\begin{ex}For $n=3$ the only equation is ($h_{ijk}:=H(e_i,e_j,e_k)$)
$$h_{234}+h_{126}=h_{135} $$
Then $H$ must be
\begin{eqnarray*}
H & = & (h_{124}e^{12}+h_{134}e^{13}+h_{234}e^{23})\wedge e^4\\
& + & (h_{125}e^{12}+(h_{234}+h_{126})e^{13}+h_{235}e^{23})\wedge e^5\\
& + & (h_{126}e^{12}+h_{136}e^{13}+h_{236}e^{23})\wedge e^6
\end{eqnarray*}

There is only one $6$-dimensional $2$-step nilpotent Lie algebra that admits no symplectic form: $\ma_5(\mathbb{R}) \times \mathbb{R}=(0,0,0,0,e^{12}+e^{34},0)$. But its center is $2$-dimensional, so it does not fit in our criteria. There are three $2$-step nilpotent Lie algebras with center of dimension $3$ or bigger and here are suitable choices of $H$ to get each one of these algebras as dual of the $6$-dimensional abelian algebra:
\begin{center}
\begin{table}[ht]
\begin{tabular}{|l|l|}
\hline
$2$-step nilpotent Lie algebra & 3-form $H$ \\ 
\hline 
\hline
$(0,0,0,0,0,e^{12})$ & $e^{126}$  \\ 
\hline 
$(0,0,0,0,e^{12},e^{13})$ &$e^{125}+e^{136}$ \\
\hline
$(0,0,0,0,e^{12},e^{13},e^{23})$& $e^{124}+e^{135}+e^{126}+e^{236}$\\
\hline
\end{tabular} \smallskip

\caption{6-dimensional 2-step nilpotent Lie algebra and the corresponding 3-form $H$}\label{table6dim}
\end{table}
\end{center}
\end{ex}
\begin{ex} For $n=4$, the $\binom{4}{3}=4$ equations are 
$$h_{127}+h_{235}=h_{136}\ \ \ \ \ h_{128}+h_{245}=h_{146}$$
$$h_{138}+h_{345}=h_{147}\ \ \ \ \ h_{238}+h_{346}=h_{247}$$

In \cite{wang2013classification} Wang, Chen and Niu classify the $8$-dimensional complex nilpotent Lie algebras with $4$-dimensional center. Ten of such Lie algebras are $2$-step nilpotent. Regarding them as $8$-dimensional real Lie algebras, we can choose suitable $H$ for all them, except one: $\ma_5(\mathbb{R}) \times \mathbb{R}^3$, which is not symplectic. We summarize these computations in Table \ref{tableH}.
\begin{center}
\begin{table}[ht]
\begin{tabular}{|l|l|}
\hline
$2$-step nilpotent Lie algebra & 3-form $H$ \\ 
\hline 
\hline
$(0,0,0,0,-e^{12}-e^{34},-e^{13},0,0)$ & $-e^{136}-e^{127}-e^{347}$  \\ 
\hline 
$(0,0,0,0,-e^{12}-e^{34},-e^{13}-e^{24},0,0)$ &$ -e^{136}-e^{246}-e^{127}-e^{347}$ \\
\hline
$(0,0,0,0,-e^{12},-e^{23},-e^{24},0)$& $-e^{235}+e^{127}-e^{348}$\\
\hline
$(0,0,0,0,-e^{12},-e^{23},-e^{34},0)$ & $-e^{125}-e^{236}-e^{347}$\\
\hline
$(0,0,0,0,-e^{12}-e^{34},-e^{23},-e^{24},0)$  & $-e^{235}+e^{127}+e^{347}-e^{248}$ \\
\hline
$(0,0,0,0,-e^{12}-e^{34},-e^{13},-e^{24},0)$ & $-e^{125}-e^{345}-e^{246}+e^{138}$\\
\hline
$ (0,0,0,0,-e^{12},-e^{23},-e^{34},-e^{24})$ & $-e^{125}-e^{236}-e^{347}-e^{248} $ \\
\hline
$(0,0,0,0,-e^{12},-e^{23},-e^{34},-e^{14})$ & $-e^{125}-e^{236}-e^{347}-e^{148}$ \\
\hline
$(0,0,0,0,-e^{12},-e^{23},-e^{34},-e^{13}-e^{24})$ & $-e^{345}-e^{126}-e^{237}+e^{138}+e^{248}$\\
\hline
\end{tabular} \smallskip

\caption{8-dimensional 2-step nilpotent Lie algebra and the corresponding 3-form $H$}\label{tableH}
\end{table}
\end{center}

\end{ex}

\section{$T$-duality on nilmanifolds}

\subsection{Structure of nilmanifolds}

A nilmanifold is a compact homogeneous manifold $E=\Lambda\bs G$ where $G$ is a simply connected nilpotent Lie group $G$ and $\Lambda$ is a discrete cocompact subgroup. We say that $E$ is $k$-step nilpotent if $G$ is so.

Recall that the exponential map $\exp:\mgg\lra G$ of a simply connected nilpotent Lie group is a diffeomorphism \cite{VAR}. A result by Malcev states that  $G$ admits a discrete cocompact subgroup (also called a lattice) if and only if there exists a basis of $\mgg$ for which the structure constants are rationals \cite{RA}. Equivalently, $\mg=\mg_0\otimes_\Q \R$ for some Lie algebra $\mg_0$ over $\Q$. Given a lattice $\Lambda$ of $G$, $\Lambda_\bullet=\span_\Z\exp^{-1}(\Lambda)$ is a discrete subgroup of the vector space $\mgg$ of maximal rank and the structure constants of a basis contained in $\Lambda_\bullet$ are rationals. Conversely, assume $\mgg$ has a basis such that the structure constants are rationals and let $\mgg_0$ be the Lie algebra over $\Q$ spanned by this basis. Then for any discrete subgroup $\Lambda_\bullet$ of maximal rank contained in $\mgg_0$, the subgroup $\lela \exp \Lambda_\bullet\rira$ of $G$ is a discrete cocompact subgroup.

Any left invariant differential form on $G$ induces a differential form on $E$. A differential form $\omega$ on $E$ is called invariant if $\alpha^*\omega$ is left invariant, where $\alpha:G\lra E$ is the quotient map. Invariant forms on $G$ are in one-to-one correspondence with alternating forms on $\mgg$, the Lie algebra of $G$. The de Rham cohomology of a nilmanifold $E=\Lambda\bs G$ can be computed from the Chevalley-Eilenberg complex of the Lie algebra $\mgg$ of $G$ \cite{NO}. In particular, any closed differential form on $E$ is cohomologous to an invariant one.
\smallskip

Below we introduce the structure of nilmanifolds as the total space of principal torus bundle over another nilmanifold.

Let $E=\Lambda\bs G$ be a nilmanifold and let $A$ be a non-trivial $m$-dimensional central normal subgroup of $G$ (always exists since $G$ is nilpotent). Hence $N=G/A$ is a nilpotent Lie group. The subgroup $\Lambda\cap A$ is a lattice in $A$ \cite{RA} and $T=(\Lambda\cap A)\bs A$ is an $m$-dimensional torus. Since $A\subset Z(G)$, the center of $G$, one has a right action of $T$ on $E$ 
$$ x\cdot a =\Lambda g\cdot  (\Lambda\cap A) z=\Lambda gz\in E,\qquad \mbox{for  }x=\Lambda g\in E,\; a=(\Lambda\cap A )z\in T.$$ The quotient space $M=E/T$ is diffeomorphic to $\Gamma\bs N$ where $\Gamma=\Lambda A/A\simeq \Lambda/\Lambda \cap A$ is a discrete cocompact group of $N$, thus $M=\Gamma\bs N$ is a nilmanifold. Therefore $E$ is the total space of the principal bundle $q:E\lra M$ with fiber $T$. Given $\Lambda g\in E$, denote by $[\Lambda g]_T$ its orbit under the $T$ action, then the fiber bundle map satisfies $q([\Lambda g]_T)=\Gamma n$ where $n=gA$. For future reference we denote $(G,A,\Lambda)$ the principal fiber bundle constructed above.

Notice that two fiber bundles  $(G,A,\Lambda)$ and  $(\tilde G,\tilde A,\tilde \Lambda)$ are equivalent if and only if each of the corresponding groups in the triple are isomorphic. In fact, since $\Lambda$ is the first homotopy group of $\Lambda\bs G$, the existence of a diffeomorphism $f:\Lambda\bs G\lra \tilde \Lambda\bs \tilde G$ implies $\Lambda\simeq \tilde \Lambda$. Mostow's rigidity  theorem \cite[Theorem 3.6]{RA} asserts that this isomorphism extends to an isomomorphism between $G$ and $\tilde G$. Since $A$ and $\tilde A$ are abelian of the same dimension, we conclude the isomorphisms between them all.

Although the following result seems to be well known in the geometry community, the only proof available in the literature is for 2-step nilmanifolds which was given by Palais and Stewart. For the sake of  completeness of the presentation we include a sketch of the proof here, which is a generalization of that in \cite{PS}.

\begin{teo}\label{teo.psgen}
A connected compact differential manifold $E$ is a nilmanifold if and only if it is the total space of a principal torus bundle over a nilmanifold.
\end{teo}
\begin{proof} We already showed how a nilmanifold can be realized as the total space of a torus bundle over a nilmanifold so we focus on the converse.

Let $E$ be a compact manifold and $q:E\lra M$ a principal fiber bundle map with an $m$-dimensional torus $T$ as structure group and $M=\Gamma\bs N$ a nilmanifold. Denote $\ma$ and $\mn$ the Lie algebras of $T$ and $N$ respectively.

Let $\omega$ be a connection in $E$ and let $\Omega$ be its curvature form. Recall that all possible $\Omega$ lie inside a unique cohomology class \cite{Kob}. This is an $\ma$-valued 2-form on $E$ and since $T$ is abelian we have $\Omega=q^*\Omega_0$ with $d\Omega_0=0$. There exist an $N$-invariant closed 2-form $\Psi_0$ on $M$ and a 1-form $\theta$, both with values on $\ma$, such that  $\Omega_0-\Psi_0=d\theta$ \cite{NO}. The 2-form $\tomega=\omega-q^*\theta$ defines a connection in $E$ with curvature $\tilde\Omega=d\tilde\omega=d\omega-dq^*\theta=\Omega-q^*d\theta=q^*\Psi_0$. 
Notice that $\Psi_0$ is induced by the left  translation of a 2-form $\Psi:\mn\times\mn\lra \ma$, which is closed in $\mn$.

Each $Z\in \ma$ induces a vector field in $E$ 
and this assignment from $\ma$ to $\mathcal X(E)$ is injective, so we identify $Z\in \ma$ with its corresponding vector field in $E$. Moreover, each $X\in \mn$ induces a left invariant vector field on $N$ which projects to a vector field on $M$, we denote  $\tX$ the vector field on $E$ which is its horizontal lift with respect to $\tomega$. The following is an equality for vector fields in $E$ induced by $X,Y\in \mn$
 \begin{equation}
[\tX,\tY]=\widetilde{[X,Y]}+\tOmega(\tX,\tY)=\widetilde{[X,Y]}+\Psi_0(X,Y)=\widetilde{[X,Y]}+\Psi(X,Y).\label{corchete}
\end{equation}

Let $\mgg=\mn\oplus \ma=\mn_\Psi$ be the central extension of $\mn$ by $\Psi$ and let $G$ be the 
 simply connected nilpotent Lie group with Lie algebra $\mgg$. 
Then $A=\exp \ma$ is a closed, connected and simply connected central ideal of $G$ \cite[Theorem 3.6.2]{VAR}. Moreover $N\simeq G/A$, we denote $\beta:G\lra N$ the quotient map.

Define $\xi:\mgg\lra \mathcal X(E)$, given by $\xi(X)=\tX$
 if $X\in \mn$ and
 $\xi(Z)=Z$ if $Z\in \ma$. The fact that $[\tilde Y,Z]=0$ for any horizontal lift $\tilde Y$ of a vector field in $M$ and Eq. \eqref{corchete} 
 imply that $\xi$ is an  injective Lie algebra homomorphism.

  The map $\xi$ is an 
 infinitesimal action of $G$ on $E$ and $E$ is compact so we can lift this action \cite[Theorem 2.16.9]{VAR} to a right action of $G$ on $E$. Given $g\in G$ and $x\in E$, let $X\in \mgg$ be the unique such that $\exp\,X=g$, then the action of $g$ on $x$ is
\begin{equation} x\cdot g=\sigma_{\xi(X),x}(1),\,\mbox{ where }\sigma_{\xi(X),x} \mbox{ is the integral curve
 of }\xi(X)\mbox{ starting at }x.\label{liftaction}\end{equation}

In the rest we prove that this is a transitive action and with discrete isotropy.

The action in \eqref{liftaction} behaves as follows. If $g\in A$ then $g=\exp Z$ for some $Z\in \ma$  and  $\xi(Z)_x=\frac{d}{dt}_{|_0} (x\cdot \exp_{T} t Z)$, so we have $\sigma_{\xi(Z),x}(t)=x\cdot \exp_{T} tZ$, where $\exp_T:\ma\lra T$. Hence $x\cdot g=x\cdot \exp_{T} Z$. In particular $x\cdot g=x$ if and only if $x=x\cdot \exp_{T} Z$ and this occurs if and only if  $Z\in \Z^m$. Thus the isotropy subgroup $G_x$ of any point $x\in E$  verifies $G_x\cap A=\exp \Z^m$.

Let now $g=\exp Y$ for some $Y\in \mn$. Then $\sigma_{\xi(Y),x}$ is an horizontal curve which is the horizontal lift through $x$ of $\tau$, where $\tau(t)=\Gamma u \exp_N tY$, $\Gamma u=q(x)$. Notice that the infinitesimal vector field on $E$ generated by $Y$ by the $G$-action is the horizontal lift of the infinitesimal vector field generated by $Y$ on $M$ by the $N$-action.  Finally, let $g=\exp X$ where $X=Y+Z$, $Y\in \mn$, $Z\in\ma$. Consider $\gamma(t)= \sigma_{\xi(Y),x}(t) \cdot \exp_T t Z$. Then $\gamma(0)=x$ and for $t_0\in \R$ we have
\begin{equation}
\frac{d}{dt}_{|_{t_0}}\gamma(t)=\frac{d}{dt}_{|_{t_0}}\left( \sigma_{\xi(Y),x}(t) \cdot \exp_T t_0 Z \right)+ \frac{d}{dt}_{|_{t_0}} \left(
\sigma_{\xi(Y),x}(t_0) \cdot\exp_T t Z\right).
\end{equation}
The curve $\sigma_{\xi(Y),x}(t) \cdot \exp_T t_0 Z$ is the horizontal lift of an integral  curve  of the vector field on $M$ induced by $Y$ through the point $x\cdot \exp_T t_0 Z$, thus its  derivative at $t_0$ is the vector field $\xi(Y)$ evaluated  at the point $\sigma_{\xi(Y),x}(t_0) \cdot \exp_T t_0 Z=\gamma(t_0)$. In addition, $
\sigma_{\xi(Y),x}(t_0) \cdot\exp_T t Z$ is a curve tangent to the fiber and its derivative at $t_0$ is $Z$ evaluated at $\gamma(t_0)$. Therefore $\frac{d}{dt}_{|_{t_0}}\gamma(t)=\tilde Y_{\gamma(t_0)}+Z_{\gamma(t_0)}=\xi(X)_{\gamma(t_0)}$ and thus $\sigma_{\xi(X),x}(t)=\gamma(t)=\sigma_{\xi(Y),x}(t) \cdot \exp_T t Z$. In particular
\begin{equation}\label{action}
x\cdot g=\sigma_{\xi(Y),x}(1) \cdot \exp_T Z.
\end{equation} 
 
Let $e$ be the identity of $N$ and fix $x_0\in q^{-1}(e)$. Let $W=q^{-1}(U)$ where $U$ a neighborhood of $\Gamma \in M$, choose $w\in W$ and denote $r=q(w)$. There is some $n\in N$  such that $\Gamma \cdot n=r$ and moreover $n=\exp_N Y$ for some $Y\in \mn$. Thus $q(\sigma_{\xi(Y),x_0}(1))=r$, since $\sigma_{\xi(Y),x_0}$ is the horizontal lift of $\Gamma \exp_N tY$ through $x_0$. So there is some $a\in T$ such that $\sigma_{\xi(Y),x_0}(1)\cdot a=w$. Let $Z\in \ma$ be such that $\exp_T Z=a$, then we obtain 
$w=\sigma_{\xi(Y),x_0}(1)\cdot \exp_T Z$ which by Eq. \eqref{action} is $x_0\cdot g$ for $g=\exp (Y+Z)$. Therefore $W\subset x_0\cdot G$ and the $x_0$ orbit is open. Since the orbit is also closed we have $x_0\cdot G=E$ and the action is transitive.

In particular $E$ is diffeomorphic to $G_{x_0}\bs G$  where $G_{x_0}$ is the isotropy at $x_0$. By construction $\dim G=
 \dim N +\dim T=\dim M +\dim T=\dim E$ so the isotropy is a discrete subgroup of $G$ and $E$ is a nilmanifold.
\end{proof}

Continue with the notation in the proof. Consider the map $\alpha:G\lra E$ given by $g\mapsto x_0 \cdot  g$, and $\pi:N\lra M$ the quotient map. At this point is clear that the following is a commutative diagram \begin{equation}\begin{array}{ccc}
    G&\stackrel{\alpha}{\lra} &E\\
    \beta\downarrow&&\downarrow q\\
     N&\stackrel{\pi}{\lra} &M,
   \end{array}
 \label{groupdiagram0}\end{equation} 
and if we denote $\Lambda=G_{x_0}$ then $\beta(\Lambda)=\Gamma$, that is, $\Lambda A/A=\Gamma$. Moreover, $\Lambda\cap A=\exp_T \Z^m$ and thus $T\simeq \Lambda\cap A\bs A$. We obtain the following.

\begin{cor}\label{cor.ppalbdl}
Let $q: E\lra M$ be a principal torus bundle with $M=\Gamma\bs N$ a nilmanifold. Then there exists a simply connected nilpotent Lie group $G$, a central subgroup $A$ of $G$ and a lattice $\Lambda$ of $G$ such that $q:E\lra M$ is equivalent to $(G,A,\Lambda)$ as principal bundles.
\end{cor}

From the proof above, in the triple $(G,A,\Lambda)$ the Lie group $G$ is the central extension of $N$ by the curvature of a connection in $q:E\lra M$, $A$ is a central subgroup of $G$ of the same dimension of the fiber and $\Lambda$ is a lattice of $G$ projecting over $\Gamma$ by $\beta$. 

\subsection{$T$-duality on nilmanifolds}

Because of its natural structure of principal torus bundles, nilmanifolds are, then, a good context to work with $T$-duality as introduced in Subsection \ref{BHM}.

\medskip

Fix a principal torus bundle $q:E\lra M$ where $M$ is a nilmanifold and identify it with its corresponding triple $(G,A,\Lambda)$ given by Corollary \ref{cor.ppalbdl}. 

\begin{defi} An invariant closed 3-form $H$ on $E$ is admissible for the bundle $q:E\lra M$ if $H(X,Y,\cdot)=0$ for any $X,Y$ vector fields tangent to $A$.
\end{defi}

Admissibility of the 3-form is independent of the lattice. In fact, if $H$ is admissible for $q:E\lra M$, identified with $(G,A,\Lambda)$, then  $H$ is admissible for $(G,A,\tilde\Lambda)$ for any $\tilde\Lambda$ lattice in $G$. Clearly, $H$ being admissible for $(G,A,\Lambda)$ implies $(\mgg,\ma,H)$ is an admissible triple.

\begin{defi} Let $q:E\lra M$ be a principal torus bundle and let $H$ be an admissible closed 3-forms. A manifold $E^\vee$ together with a 3-form $H^\vee$ is said to be invariantly dual to $(E,H)$ if $E^\vee$ is the total space of a principal torus bundle over $M$, $H^\vee$ is an admissible form for this bundle and $(E^\vee, H^\vee)$ is $T$-dual to $(E,H)$ with invariant 2-form $F$ in the correspondence space. (Notice that  $E\times_M E^\vee$ is also a nilmanifold.)
\end{defi}

The following is a clear result.
\begin{pro} \label{pro.suf} Let $E$ and $E^\vee$ be torus bundles over the same nilmanifold $M$ and let $H$ and $H^\vee$ be corresponding admissible closed 3-forms for each bundle. If $(E, H)$ and $(E^\vee, H^\vee)$ are invariant $T$-duals then $(\mgg,\ma,H)$ and $(\mgg^\vee,\ma^\vee,H^\vee)$ are dual triples.
\end{pro}

The next example is due to Mathai and Rosenberg \cite{MR} (see also \cite{Mat}). 
\begin{ex}\label{mathai}
The 3-dimensional Heisenberg Lie group $H_3$ can be realized on $\R^3$ with its usual differentiable structure together with the multiplication law
$$(x,y,z)\cdot(x',y',z')=\left(x+x',y+y',z+z'+\frac12 (xy'-yx')\right). $$ A basis of left invariant vector fields satisfying the relation $[X_1,X_2]=X_3$ and its corresponding dual basis is given, at $(x,y,z)\in \R^3$, by
$$\begin{array}{rclrcl}
X_1&=&\partial_x-\frac{y}2\partial_z&\qquad\omega_1&=&dx\\
X_2&=&\partial_y+\frac{x}2\partial_z&\qquad\omega_2&=&dy\\
X_3&=&\partial_z&\qquad\omega_3&=&dz+\frac12(ydx-xdy).
\end{array} $$
Finally, the exponential map $\exp:\mh_3\lra H_3$ is the identity, that is,  $\exp(aX_1+bX_2+cX_3)=(a,b,c)$ for all $a,b,c\in \R$.

As usual, identify the tangent bundle of  $H_3\times H_3$ at a point $(g,g')$ with the sum of the tangent bundle at $g$ and $g'$; similarly for the cotangent bundle. Denote $p$ and $p^\vee$ the projections over the first and second coordinates, respectively. Restrict these projections to the Lie subgroup $C=\{(g,g')\in H_3\times H_3: g=(x,y,z),\, g'=(x,y,z'),\ x,y,z,z'\in\R\}$, then we have $\Omega_i:=p^*\omega_i=(\omega_i,\omega_i)$ for $i=1,2$, $p^*\omega_3=(\omega_3,0)=:\Omega_3$ and $p^{\vee*}\omega_3=(0,\omega_3)=:\tilde\Omega_3$, where $d\Omega_3=d\tilde\Omega_3=\Omega_1\wedge \Omega_2$. These are invariant 1-forms in $C$.

For each $k\in \N$ let $\Lambda_k$ be the discrete cocompact subgroup $\Lambda_k=\Z\times \Z\times \frac{1}{2k}\Z$ and denote $E_k=\Lambda_k\bs H_3$. This is the total space of a principal $S^1$ bundle over $T^2$ and if $A$ is the center of $H_3$, then $(H_3,A,\Lambda_k)$ is the associated triple to this bundle.

Let $T^3=\Z^3\bs \R^3$ be the 3-torus and let $f_k:E_k\lra T^3$, $f_k(\Lambda_k (x,y,z))=\Z^3(x,y,2kz)$; this is a well defined differentiable mapping. Let $vol_k=f_k^*(vol)$ where $vol=dx\wedge dy\wedge dz$ is the canonical volume form in $T^3$, then $vol_k=2k \omega^1\wedge \omega^2\wedge \omega^3$ and it is clearly an invariant closed 3-form in $E_k$;  moreover $vol_k$ is admissible for $(H_3,A,\Lambda_j)$, for any $j$. For each $j,k\in \N$, the pair $(E_k, 2jvol_k)$ is invariantly $T$-dual to $(E_j,2kvol_j)$, as we establish below.

The subgroup $\Lambda_k\times \Lambda_j\cap C$ is a lattice of $C$ since the  nilmanifolds fiber over $M=T^2$. Moreover $W:=\Lambda_k\times \Lambda_j\cap C\bs C$ is the correspondence space for the  bundle maps $q$ and $q^\vee$. 
The invariant forms $\Omega_i$ and $\tilde\Omega_3$ are induced to $W$ so we take $F=4jk\tilde\Omega_3\wedge \Omega_3$ which is an invariant non-degenerate 2-form in $W$. We now have 
\begin{eqnarray*}
dF&=&4jk(\Omega_1\wedge \Omega_2\wedge \Omega_3-\tilde\Omega_3\wedge\Omega_1\wedge \Omega_2)\\
&=&2j(2k \Omega_1\wedge \Omega_2\wedge \Omega_3)-2k(2j\Omega_1\wedge \Omega_2\wedge\tilde\Omega_3)\\
&=&p^*(2jvol_k)-p^{\vee*}(2kvol_j),
\end{eqnarray*}
so the (invariant) $T$-duality is proved.\end{ex}

\begin{remark}
The previous example can be extended. For any $j,k \in \N$ the pair $(E_k, H)$ is invariantly $T$-dual to $(E_j,H)$, if $H$ is a non-zero left-invariant closed 3-form in $H_3$ (induced to the nilmanifold).

In fact both $E_k$ and $E_j$ are torus bundles over $T^2$, and because of dimensionality reasons, $H=\lambda \omega^1 \wedge \omega^2\wedge \omega^3$ for some $\lambda\neq 0$. As above, the 2-form $F=\lambda \tOmega_3\wedge \Omega_3$ gives the duality.

Notice that $E_j$ is not diffeomorphic to $E_{j'}$ if $j\neq j'$ since $\Lambda_j$ is not isomorphic to $\Lambda_{j'}$.
\end{remark}

After this remark, uniqueness of dual pairs in this context of nilmanifolds is not expected. We address then the existence question.

Proposition \ref{pro.suf} does not guarantee existence of the dual of a given torus bundle $q:E\lra M$ with an admissible $H$. Instead, it states that if such a dual exists, say $E^\vee$, and this is an invariant duality then $E^\vee$ is a quotient of the simply connected nilpotent Lie group $G^\vee$ associated to the Lie algebra $\mgg^\vee$ described in Theorem \ref{teo.infdual}. Also, $H^\vee$ and $F$ are the ones given in the same theorem. In particular, if this is the case, $G^\vee$ would admit a lattice.
\medskip

To finish the paper we prove an existence result under the assumption that $H$ satisfies a rational condition which, in particular, warrants the existence of lattices in $G^\vee$.

We continue with $q:E\lra M$ a principal torus bundle over the nilmanifold $M$ which is equivalent to the bundle $(G,A,\Lambda)$. The set $\Lambda_\bullet=\span_\Z\exp^{-1}(\Lambda)$ is a discrete subgroup of $\mgg$ of maximal rank and $\Lambda_\bullet\cap \ma$ is a discrete subgroup of $\ma$. The quotient map $\beta:G\lra N$ projects $\Lambda$ to the lattice $\Gamma$ in $N$. The discrete subgroup of $\mn$ corresponding to $\Gamma$ is $\Gamma_\bullet=\span_\Z\exp^{-1}(\Gamma)$ and satisfies $\beta_*(\Lambda_\bullet)=\Gamma_\bullet$, where here $\beta_*$ is the differential of $\beta$ at the identity.

The set $\mgg_0=\span_\Q\exp^{-1}(\Lambda)$ is a Lie algebra over $\Q$ and $\mgg=\mgg_0\otimes_\Q \R$. We may choose a basis $\mathcal B=\{Y_1,\ldots,Y_s,X_1,\ldots,X_m\}$ of $\mgg_0$ such that $\Lambda_\bullet\cap \ma=\Z X_1+\cdots+\Z X_m$; the structure coefficients in this basis are rational numbers.

\begin{teo} \label{teo.global} Let $q:E\lra M$ be a principal torus bundle, identified to $(G,A,\Lambda)$,  and let $H$ be an admissible closed 3-form. Let $(\mgg^\vee,\ma^\vee,H^\vee)$ be the dual triple to $(\mgg,\ma,H)$.

Assume $\iota_{X_i}H(Y_j,Y_k)\in \Q$ for all $i,j,k$ in a basis as above. Then there exists a lattice $\Lambda^\vee$ in $G^\vee$ such that $(G^\vee, A^\vee,\Lambda^\vee)$ is invariantly $T$-dual to $q:E\lra M$.
\end{teo}

\begin{proof} Let  $Z_1,\ldots, Z_m$ be a basis of $\ma^\vee$ so that $\mathcal B^\vee=\{Y_1,\ldots,Y_s,Z_1,\ldots,Z_m\}$ is a basis of $\mgg^\vee=\mn\oplus \ma^\vee$. The Lie brackets of these basic elements are
$$[Y_i,Y_j]_{\mgg^\vee}= [Y_i,Y_j]_{\mn}+\iota_{X_1}H(Y_i,Y_j)Z_1+\cdots + \iota_{X_m}H(Y_i,Y_j)Z_m.$$ Thus the structure constants corresponding to the basis $\mathcal B^\vee$ are rational. Consider the subset $\Lambda_\bullet^\vee=\Gamma_\bullet+\Z Z_1+\cdots+\Z Z_m$ of $\mgg^\vee$; this 
is a discrete (aditive) subgroup of $\mg^\vee$ since $\Gamma_\bullet$ is such a subgroup of $\mn$ and $\mg^\vee=\mn\oplus\ma^\vee$. Moreover, $\Gamma_\bullet$ is of maximal rank $m+s$ and it is contained in $\mgg_0=\span_\Q \mathcal B^\vee$. According to \cite[Theorem 2.12]{RA} the subgroup generated by $ \exp(\Lambda_\bullet^\vee)$, which we shall denote $\Lambda^\vee$, is a discrete subgroup of $G^\vee$ and $\Lambda^\vee\bs G^\vee$ is compact.

We need to show that $A^\vee\Lambda^\vee/A^\vee\simeq \Gamma=A\Lambda/A$ in order to prove that $\Lambda^\vee\bs G^\vee$ is in fact the total space of a torus bundle over $\Gamma\bs N$. 

Denote $\beta^\vee:G^\vee\lra N$ the quotient map and its differential at the identity by $\beta^\vee_*$. Because of the definition and the commutative diagram 
$$\begin{array}{rcl}
\mgg^\vee &\stackrel{\exp}{\lra} &G^\vee\\
\beta^\vee_*\downarrow&&\downarrow \beta^\vee\\
\mn&\stackrel{\exp}\lra &N,
\end{array}$$
we have that $\beta^\vee(\Lambda_\bullet^\vee)=\Gamma_\bullet$ which implies $\beta^\vee_*(\exp \Lambda_\bullet^\vee)=\exp\Gamma_\bullet$ and, since $\beta^\vee$ is an homomorphism $\beta^\vee(\lela\exp \Lambda_\bullet^\vee\rira)=\lela\exp\Gamma_\bullet\rira$, or equivalently, $A^\vee\Lambda^\vee/A^\vee=\beta^\vee(\Lambda^\vee)=\Gamma$. 
\end{proof}

\begin{ex}\label{KTTglobal} Let $G=H_3\times \R$ modeled on $\R^4$ using the coordinates $(x,y,z,t)$ where $(x,y,z)\in H_3$ given in Example \ref{mathai} and $t$ is the coordinate on $\R$. We have $X_i$, $i=1,\cdots,4$ a basis of left invariant vector fields such that the only nonzero bracket on this basis is $[X_1,X_2]=X_3$. Let $\Lambda$ be the discrete subgroup of $G$ of points with each coordinate an integer, $\Lambda=\Z^4$. The manifold $E=\Lambda\bs G$ is known as the Kodaira-Thurston nilmanifold. Let $A=Z(G)$ and $H=0$ then, clearly, $H$ is an admissible closed 3-form for $(G,A,\Lambda)$ and $H$ satisfies the rational condition in Theorem \ref{teo.global}.
Moreover, in Example \ref{KTT} we have seen that $(\mh_3\oplus\R,\ma,0)$ is dual to $(\ma_4,\ma_2,H^\vee)$ with $\ma_i$ abelian Lie algebras of dimension $i$, and $H^\vee\neq 0$. 
Thus there exists a lattice $\Lambda^\vee\subset \R^4$ such that  $(E,0)$ is $T$-dual to $T^4=(\Lambda^\vee\bs \R^4,H^\vee)$. In particular, there is a bijection between invariant generalized complex structures on $E$ and those on $T^4$, with the corresponding $3$-forms.\end{ex}

\begin{remark}
\begin{enumerate}
\item[a)] $T$-duality between the Kodaira-Thurston nilmanifold and the 4-dimen\-sional torus, and the correspondence of generalized complex structures,  has been studied by Aldi and Heluani \cite{A-H}, via the understanding of the {\em complex structures} on the 8-dimensional product space $E\times T^4$. 

\item[b)] The {\em homological mirror symmetry} between the Kodaira-Thurston nilmanifold and $T^4$ was recently established by Abouzaid, Auroux, Katzar\-kov and Orlov \cite{AAKO} (see also \cite{Ab-ICM}).

\item[c)] In \cite{CG} it was already noticed that every 2-step nilmanifold with vanishing 3-form is $T$-dual to a torus with
nonvanishing 3-form. 

\end{enumerate}
\end{remark}\bigskip

\bibliographystyle{plain}

\begin{thebibliography}{10}

\bibitem{Ab-ICM}
M.~Abouzaid.
\newblock {Family Floer cohomology and mirror symmetry}.
\newblock arXiv:1404.2659, 2014.

\bibitem{AAKO}
M.~Abouzaid, D.~Auroux, L.~Katzarkov, and D.~Orlov.
\newblock {Homological mirror symmetry for Kodaira-Thurston manifolds.}
\newblock in preparation.

\bibitem{A-H}
M.~Aldi and R.~Heluani.
\newblock On a complex-symplectic mirror pair.
\newblock arXiv:1512.02258, 2015.

\bibitem{BEM}
P.~Bouwknegt, J.~Evslin, and V.~Mathai.
\newblock T-duality: topology change from {H}-flux.
\newblock {\em Comm. Math. Phys.}, 249(2):383--415, 2004.

\bibitem{BHM}
P.~Bouwknegt, K.~Hannabuss, and V.~Mathai.
\newblock T-duality for principal torus bundles.
\newblock {\em J. High Energy Phys.}, (3), 2004.

\bibitem{BHM06}
P.~{Bouwknegt}, K.~{Hannabuss}, and V.~{Mathai}.
\newblock {Nonassociative tori and applications to $T$-duality.}
\newblock {\em {Commun. Math. Phys.}}, 264(1):41--69, 2006.

\bibitem{BRS}
U.~Bunke, P.~Rumpf, and T.~Schick.
\newblock The topology of t-duality for tn-bundles.
\newblock {\em Reviews in Mathematical Physics}, 18(10):1103--1154, 2006.

\bibitem{BS}
U.~Bunke and T.~Schick.
\newblock On the topology of t-duality.
\newblock {\em Reviews in Mathematical Physics}, 17(01):77--112, 2005.

\bibitem{CdB}
L.~{Cagliero} and V.~{del Barco}.
\newblock {Nilradicals of parabolic subalgebras admitting symplectic
  structures.}
\newblock {\em {Differ. Geom. Appl.}}, 46:1--13, 2016.

\bibitem{CG}
G.~R. Cavalcanti and M.~Gualtieri.
\newblock Generalized complex geometry and t-duality.
\newblock {\em CRM Proc. Lecture Notes, ``A celebration of the mathematical
  legacy of Raoul Bott''}, 50:341--365, 2011.

\bibitem{wang2013classification}
H.~Chen, Y.~Niu, and Y.~Wang.
\newblock Classification of 8-dimensional nilpotent lie algebras with
  4-dimensional center.
\newblock {\em J. Math. Res. Appl.}, 33(6):699--707, 2013.

\bibitem{CLP}
R.~{Cleyton}, J.~{Lauret}, and Y.~S. {Poon}.
\newblock {Weak mirror symmetry of Lie algebras.}
\newblock {\em {J. Symplectic Geom.}}, 8(1):37--55, 2010.

\bibitem{ABDF}
L.~de~Andr\'es, M.L. Barberis, I.~Dotti, and M.~Fern\'andez.
\newblock Hermitian structures on cotangent bundles of four dimensional
  solvable lie groups.
\newblock {\em Osaka J. Math.}, 44(4):765--793, 2007.

\bibitem{dBG}
V.~del Barco and L.~Grama.
\newblock {On generalized $G_2$-structures and $T$-duality}.
\newblock arXiv:1709.09154v1, 2017.

\bibitem{gualtieri}
M.~Gualtieri.
\newblock Generalized complex geometry.
\newblock {\em Ann. of Math.}, 174:75--123, 2011.

\bibitem{Kob}
S.~{Kobayashi}.
\newblock {Principal fibre bundles with the 1-dimensional toroidal group.}
\newblock {\em {Tohoku Math. J. (2)}}, 8:29--45, 1956.

\bibitem{Mat}
V.~Mathai.
\newblock T-duality: a basic introduction.
\newblock {\em \url{  http://users.uoa.gr/~iandroul/GAP%202012%20WEBPAGE/GAP2012_Varghese1.pdf}},
  GAP 2012 Minicourse, 2012.

\bibitem{MR}
V.~Mathai and J.~Rosenberg.
\newblock {T-duality for circle bundles via noncommutative geometry}.
\newblock {\em Adv. Theor. Math. Phys.}, 18(6):1437--1462, 2014.

\bibitem{NO}
K.~Nomizu.
\newblock On the cohomology of compact homogeneous spaces of nilpotent {L}ie
  groups.
\newblock {\em Ann. of Math. (2)}, 59:531--538, 1954.

\bibitem{Ovando}
G.~Ovando.
\newblock {Four dimensional symplectic Lie algebras}.
\newblock {\em Beiträge Algebra Geom.}, 47:419--434, 2006.

\bibitem{PS}
R.S. {Palais} and T.E. {Stewart}.
\newblock {Torus bundles over a torus.}
\newblock {\em {Proc. Am. Math. Soc.}}, 12:26--29, 1961.

\bibitem{RA}
M.S. Raghunathan.
\newblock {\em Discrete subgroups of {L}ie Groups}.
\newblock Springer, 1972.

\bibitem{RG}
E.~Remm and M.~Goze.
\newblock Symplectic structures on 2-step nilpotent lie algebras.
\newblock arXiv:1510.082121, 2015.

\bibitem{SM}
L.~San~Martin.
\newblock {\em Álgebras de {L}ie}.
\newblock Editora da Unicamp, 1999.

\bibitem{Sc}
M.~{Schottenloher}.
\newblock {\em {A mathematical introduction to conformal field theory. Based on
  a series of lectures given at the Mathematisches Institut der Universit\"at
  Hamburg. 2nd revised ed.}}
\newblock Berlin: Springer, 2nd revised ed. edition, 2008.

\bibitem{SYZ}
A.~Strominger, S.T. Yau, and E.~Zaslow.
\newblock Mirror symmetry is {T}-duality.
\newblock {\em Nuclear Phys. B}, 479(1-2):243--259, 1996.

\bibitem{VAR}
V.~S. Varadarajan.
\newblock {\em Lie groups, {L}ie algebras, and their representations}, volume
  102 of {\em Graduate Texts in Mathematics}.
\newblock Springer-Verlag, New York, 1984.
\newblock Reprint of the 1974 edition.

\end{thebibliography}

\end{document}